\documentclass[12pt]{amsart}
\usepackage{amssymb,latexsym}

\newtheorem{thm}{Theorem}
\newtheorem{conjecture}{Conjecture}

\newtheorem{lem}[thm]{Lemma}

\newtheorem{bc}[thm]{Brown's Filtration Criterion}

\def\bbz{\mathcal{O}_K}
\def\bbq{K}

\def\bbr{\mathbb{R}}

\def\bfA{\mathbf{A}}
\def\bfG{\mathbf{G}}
\def\bfH{\mathbf{H}}
\def\bfM{\mathbf{M}}
\def\bfP{\mathbf{P}}
\def\bfQ{\mathbf{Q}}
\def\bfS{\mathbf{S}}
\def\bfT{\mathbf{T}}
\def\bfU{\mathbf{U}}
\def\SL{{\rm{\mathbf{SL}}}}

\def\G{\Gamma}
\def\lq{\bbq((t^{-1}))}
\def\slp{{\mathbf{SL}_n(\mathcal{O}_K[t])}}
\def\sll{{\mathbf{SL}_n(K((t^{-1})))}}
\def\slr{{\mathbf{SL}_n(K[t])}}
\def\slt{{\mathbf{SL}_n(K[[t^{-1}]])}}

\title[polynomial points of simple groups over number fields]
{On presentations of integer  polynomial points of simple groups over number fields}

\thanks{The writing of this paper was supported in part by NSF Grant No. DMS-0905891. }

\author{Amir Mohammadi \& Kevin Wortman}

\begin{document}

\maketitle

 In this paper we prove the following

\begin{thm}\label{t;main}
Let $K$ be a number field and let $\bbz$ be its ring of integers. Let $\bfG$ be a connected, noncommutative, absolutely  almost simple algebraic $\bbq$-group. If the $\bbq$-rank of $\bfG$ equals $2$, then $\bfG(\bbz[t])$ is not finitely presented.
\end{thm}

Actually, we will prove a slightly stronger version of Theorem~\ref{t;main} by showing that if $\bfG(\bbz[t])$ is as in Theorem~\ref{t;main}, then $\bfG(\bbz[t])$ is not of type $FP_2$.

\subsection{Related results}
Krsti\'{c}-McCool proved that  ${\rm GL_3}(A)$ is not finitely presented if there is an epimorphism from $A$ to $F[t]$ for some field $F$ \cite{K-M}. 

Suslin proved that ${\rm SL _n}(A[t_1,\ldots , t_k])$ is generated by elemetary matrices if $n \geq 3$, $A$ is a regular ring, and $K_1(A)\cong A^\times$  \cite{Su}. Grunewald-Mennicke-Vaserstein proved that ${\rm Sp _{2n}}(A[t_1,\ldots , t_k])$ is generated by elementary matrices if $n \geq 2$ and $A$ is a Euclidean ring or a local principal ideal ring \cite{G-M-V}.

In Bux-Mohammadi-Wortman, it's shown that ${\rm SL_n}(\mathbb{Z}[t])$ is not of type $FP_{n-1}$ \cite{BMW}. The case when $n=3$  is a special case of Theorem~\ref{t;main}.

While most of the results listed above allow for more general rings than $\mathcal{O}_K[t]$, the result of this paper, and the techniques used to prove it, are distinguished by their applicability to a class of semisimple groups that extends beyond special linear and symplectic groups.

\section{Preliminary and notation}~\label{sec;notation}

Throughout the remainder, we let $\bf G$ be as in Theorem~\ref{t;main} and we let $\G=\bfG(\bbz[t]).$

Let $L$ be an algebraically closed field containing $\lq$ fixed once and for all. In the the sequel the Zarsiki topology is defined with this fixed algebraically closed field in mind.  

Let $\bfS$ be a maximal $\bbq$-split torus of $\bfG.$ Let $\{\alpha,\beta\}$ be a set of simple $\bbq$-roots for $(\bfG,\bfS),$ and define $\bfT=(\ker(\alpha))^\circ,$ the connected component containing the identity.

 Let $\bfP$ be a maximal $\bbq$-parabolic subgroup of $\bfG$ that has $Z_\bfG(\bfT)$  as a Levi subgroup where $Z_\bfG(\bfT)$ denotes the centralizer of $\bfT$ in $\bfG.$
 Let $\bfU$ be unipotent radical of $\bfP.$ We have $\bfP={\bfU}Z_\bfG(\bfT)$. We can further write 
$$\bfP=\bfU\bfH\bfM\bfT$$
where  $\bfH\leq Z_\bfG(\bfT)$ is a simple $\bbq$-group of $\bbq$-rank 1 and $\bfM$ is a $\bbq$-anisotropic torus contained in the center of $ Z_\bfG(\bfT)$.

If $x\in\lq$ is algebraic over $\bbq$ then $x\in\bbq$, hence $\bfG$ has $\lq$-rank $2$ as well and $\bfP$ is a $\lq$-maximal parabolic of $\bfG$. It also follows that $\bf H$ has $\lq$-rank $1$ and that $\bf M$ is $\lq$-anisotropic.

We let $G,$ $S,$ $P,$ $U,$ $M,$ $H$ and $T$ denote the $\lq$-points of $\bfG,$ $\bfS,$ $\bfP,$ $\bfU,$ $\bfM,$ $\bfH,$ and $\bfT,$ respectively.

Let $X$ denote the Bruhat-Tits building associated to $G.$ This is a $2$-dimensional simplicial complex, and the apartments (maximal flats) correspond to maximal $\lq$-split tori.  

We fix once and for all a $\bbq$-embedding of $\bfG$ in some $\SL_n.$ Using this embedding we realize $\bfG(\bbq[t])$ and $\G$ as subgroups of $\SL(\bbq[t])$ and $\SL(\bbz[t])$ respectively. This embedding also gives an isometric embedding of $X$ into $\tilde{A}_{n-1},$ the building of $\SL_n(\lq)$;  see~\cite{La}.    

\section{Stabilizers of the $\G$-action on its Euclidean building}\label{sec;stab}  

\begin{lem}\label{l;stab} If $X$ is the Euclidean building for $G$, then the $\G$ stabilizers of cells in $X$ are $FP_m$ for all $m$. 
\end{lem}

\begin{proof}
We first recall the proof of~\cite[Lemma 2]{BMW}. Let $x_0 \in \tilde{A}_{n-1}$ be the vertex stabilized by $\slt$. We
denote a diagonal matrix in $\mathbf{GL_n}(K((t^{-1})))$
with entries $s_1, s_2,..., s_n \in K((t^{-1}))^\times$
by $D(s_1,s_2,...,s_n)$, and we let $\mathfrak{S} \subseteq \tilde{A}_{n-1}$ be
the sector based at $x_0$ and containing vertices of the form
$D(t^{m_1}, t^{m_2},...,t^{m_n})x_0$ where each $m_i \in
\mathbb{Z}$ and $m_1 \geq m_2 \geq ... \geq m_n$.

The sector $\mathfrak{S}$ is a fundamental domain for the action
of $\mathbf{SL_n}(K[t])$ on $\tilde{A}_{n-1}$ (see \cite{So}). In particular, for any vertex $z
\in \tilde{A}_{n-1}$, there is some $h'_z\in \mathbf{SL_n}(K[t])$ and some integers $m_1 \geq
m_2 \geq ... \geq m_n$ with $z=h'_zD_z(t^{m_1},
t^{m_2},...,t^{m_n})x_0$. We let $h_z=h'_zD_z(t^{m_1},
t^{m_2},...,t^{m_n}).$

For any $N \in \mathbb{N}$, let $W_N$ be the $(N+1)$-dimensional
vector space $$W_N=\{\,p(t) \in \mathbb{C}[t] \mid
\text{deg}\big(p(t)\big) \leq N \}$$ which is endowed with the
obvious $K-$structure. If $N_1,\cdots,N_{n^2}$ in
$\mathbb{N}$ are arbitrary then let
$$\mathbf{G}_{\{N_1,\cdots,N_{n^2}\}}=\{\mathbf{x}\in\prod_{i=1}^{n^2}W_{N_i}|
\hspace{1mm}\mbox{det}(\mathbf{x})=1\}$$ where
$\mbox{det}(\mathbf{x})$ is a polynomial in the coordinates of
$\mathbf{x}.$ To be more precise this is obtained from the usual
determinant function when one considers the usual $n\times n$
matrix presentation of $\mathbf{x},$ and calculates the
determinant in $\mathbf{Mat}_n(\mathbb{C}[t]).$

For our choice of vertex $z \in \tilde{A}_{n-1}$ above, the stabilizer of $z$ in
$\sll$ equals $h_z \slt h_z^{-1}.$ And with our fixed choice of
$h_z$, there clearly exist some $N^z_i \in \mathbb{N}$ such that
the stabilizer of the vertex $z$ in $\slr$ is
$\mathbf{G}_{\{N^z_1,\cdots,N^z_{n^2}\}}(K)$.
Furthermore, conditions on $N^z_i$ force a group structure on
$\mathbf{G}_z=\mathbf{G}_{\{N^z_1,\cdots,N^z_{n^2}\}}.$ Therefore,
the stabilizer of $z$ in $\slr$ is the $K$-points of the
affine $K$-group $\mathbf{G}_z$, and the stabilizer of
$z$ in $\slp$ is $\mathbf{G}_z(\mathcal{O}_K)$.

Let $\sigma$ be a cell in $\tilde{A}_{n-1}$. The action of $\slr$ on $\tilde{A}_{n-1}$ is type preserving, so if $\sigma
\subset \mathfrak{S}$ is a simplex with vertices $z_1, z_2,...,z_m$, then
the stabilizer of $\sigma$ in $\slp$ is 
$$\big(\mathbf{G}_{z_1} \cap \cdots \cap\mathbf{G}_{z_m}\big)(\mathcal{O}_K)$$
Which implies that the stabilizer of $\sigma$ in $\G$ is 
$\mathbf{G}_{\sigma} (\bbz) $ where
$\mathbf{G}_{\sigma} =\bfG\cap\mathbf{G}_{z_1} \cap \cdots \cap\mathbf{G}_{z_m}$.

If $\psi \subset X$ is a cell, then we let $\sigma _1, \ldots ,\sigma_k$ be simplices of $\tilde{A}_{n-1}$ such that their union contains $\psi$, and such that their union is contained in the union of any other set of simplices of $\tilde{A}_{n-1}$ that contains $\psi$. 

The group $\G$ may not act on $X$ type-preservingly, but the stabilizer of $\psi$ in $\G$ will contain a finite index subgroup that fixes $\psi$ pointwise. Because $\G$ does act type-preservingly on $\tilde{A}_{n-1}$, we have that the stabilizer of $\psi$ in $\Gamma$ 
contains
$$ \big( \mathbf{G}_{\sigma_1}  \cap \cdots \cap \mathbf{G}_{\sigma_k} \big)(\bbz)$$
as a finite index subgroup. This is an arithmetic group, and Borel-Serre \cite{B-S} proved that any such group is $FP_m$ for all $m.$

\end{proof}

\section{An unbounded ray in $\G \backslash X$}

The group $\G$ does not act cocompactly on $X$. Our next lemma is a generalization of Mahler's compactness criterion, and it will help us identify a ray in $X$ whose projection to $\G \backslash X$ is proper. Our proof is similar to~\cite[Lemma 11]{BMW}.

\begin{lem}\label{l;mahler} 
If $e \in X$, $a \in G$, $u \in \Gamma$ is nontrivial, and $a^{-n}ua^{n}\rightarrow 1$ as $n\to \infty$, then $\{\Gamma a^ne: n\geq0\}\subset\Gamma\backslash X$ is unbounded.
\end{lem}

\begin{proof}
Since $G$ acts on $X$ with bounded point stabilizers, it suffices to show that  $\{\Gamma a^n: n\geq0\}\subset\Gamma\backslash G$ is unbounded. 

If $\{\Gamma a^n: n\geq0\}$ is bounded, then it is contained in a set $\Gamma B $ where $B \subset G$ is a bounded set. Thus, for any $a^n$, we have $a^n=\gamma b$ for some $\gamma \in \G$ and $b \in B$. Hence $a^{-n}ua^{n}=b^{-1}\gamma ^{-1}u\gamma b$.

Because $u$ is nontrivial, $\gamma ^{-1}u\gamma \in \G -1$ is bounded away from $1$, and thus 
$b^{-1}\gamma ^{-1}u\gamma b$ is bounded away from $1$. That's a contradiction.
\end{proof}

\section{An unbounded semisimple element in ${\bf H}(\mathcal{O}_K[t])$}\label{sec;sem}

Recall that $\bfH$ has $\lq$-rank 1 (and $\bbq$-rank 1), hence the Bruhat-Tits building of $H$, which will be denoted by $X_H,$ is a tree. Let $\bfS'$ be a maximal $\bbq$-split, thus $\lq$-split, torus of $\bfH$ and let $\bfQ^+$ and $\bfQ^-$ be opposite $\bbq$-parabolic subgroups of $\bfH$ with Levi subgroup $Z_\bfH(\bfS').$ 

We denote the unpotent radical of $\bfQ^\pm$ as $R_u(\bfQ^\pm)$, and we let $Q^\pm=\bfQ^\pm(\lq)$, $R_u(Q^\pm)=R_u(\bfQ^\pm)(\lq)$, and $S'=\bfS'(\lq).$

See~\cite[Proposition 25]{Se} for the next lemma.

\begin{lem}\label{l;aniso-torus}
Let $u^+\in R_u(Q^+)$ and $u^-\in R_u(Q^-)$ and let $F^\pm={\rm{Fix}}_{X_H}(u^\pm).$ Assume that $F^+\cap F^-=\emptyset.$ Then $u^+u^-$ is a hyperbolic isometry of $X_H.$
\end{lem}

\begin{proof}

Let $x$ be the midpoint between $F^+$ and $F^-$.
Let $p_1$ be the path between $x$ and $F^+$ and let $p_2$  be the path between $x$ and $F^-$, and let $\psi$ be an edge containing $x$, contained in $p_1 \cup p_2$, not contained in $p_2$, and oriented towards $F^+$.

Notice that $u^-p_2 \cup p_2$ is an embedded path between $x$ and $u^-x$ and that $p_1 \cup u^+p_1 \cup u^+p_2 \cup u^+u^-p_2$ is an embedded path between $x$ and $u^+u^-x$. The edge $u^+u^-\psi$ is a continuation of the latter path that is oriented away from from both $u^+u^-x$ and $x$.

If $u^+u^-$ is elliptic, then it fixes the midpoint of the path between $x$ and $u^+u^-x$ and maps $\psi$ to an oriented edge pointed towards $x$. Therefore, $u^+u^-$ is hyperbolic.

\end{proof}

\begin{lem}
There exists elements $u^\pm \in R_u(\bfQ^\pm)(\bbz[t])$ of arbitrarily large norm.
\end{lem}

\begin{proof}
After perhaps replacing $\alpha$ with $2 \alpha$, there is a root group ${\bf U}_\alpha \leq R_u(\bfQ^\pm)$ and a $\bbq$-isomorphism of algebraic groups
$f: \mathbb{A}^k \rightarrow {\bf U}_\alpha$ 
 for some affine space $\mathbb{A}^k$.
 
 The regular function $f$ is defined by polynomials $f_i \in \bbq [x_1, \ldots , x_k]$. Because $f$ maps the identity element to the identity element, each $f_i$ has a constant term of $0$.
  
 The field of fractions of $\mathcal{O}_K$ is $K$. We let $N$ be the product of the denominators of the coefficients of the $f_i$. Then the image under $f$ of the points $(Nt^j, \ldots , Nt^j)$ forms an unbounded sequence in $j$ of points in ${\bf U}_\alpha (\bbz[t])$.
 
\end{proof}

\begin{lem}\label{l;b}
There exists a hyberbolic isometry $b \in \bfH(\bbz[t])$ of the tree $X_H$.
\end{lem}

\begin{proof}
Let $\ell ' \subseteq X_H$ be the geodesic corresponding to $S'$, and choose $u^\pm \in R_u(\bfQ^\pm)(\bbz[t])$ of sufficient norm such that $\ell ' \cap F^+$ is disjoint from $\ell ' \cap F^-$.
Since $F^+$ and $F^-$ are convex, and $\ell ' -( F^+ \cup F^-)$ is the geodesic between them, it follows that $ F^+\cap F^- =\emptyset$. Now apply Lemma~\ref{l;aniso-torus}.
\end{proof}

\section{Construction of cycles in $X$ near $\Gamma$}\label{sec;cyc}

Let $b \in \bfH(\bbz[t])$ be as in Lemma~\ref{l;b}, and let $\bfS''$ be the $\lq$-split one dimensional torus corresponding to the axis of $b$ in $X_H$. Define the $\lq$-split torus $\bfA=\langle\bfS'',\bfT\rangle \leq \bfP$ and let $A=\bfA(\lq).$ Let $\mathcal{A}$ denote the apartment in $X$ corresponding to $A.$ 

Recall that any unbounded element $a \in T$ translates $\mathcal{A}$, and that the axis for the translation is any geodesic in $\mathcal{A}$ that joins $P$ with its opposite parabolic $P^{op},$ as usual $P^{op}={\bf P}^{op}(\lq)$ where ${\bf P}^{op}$ is the oppositie parabolic containing $Z_\bfG(\bfT).$

Note that $b$ acts by translation on $\mathcal{A}$. In fact, $b$ translates orthogonal to any geodesic in $\mathcal{A}$ that joins $P$ with $P^{op}$.
Indeed, choose an element $w$ of the Weyl group with respect to $\bf A$ that reflects through a geodesic joining $\bf P$ and ${\bf P}^{op}$. Thus $w$ fixes both parabolic groups, and their common Levi subgroup, and hence $\bf H$. Since ${\bf S' }= {\bf A} \cap {\bf H}$, $w$ fixes $\bf S'$ and thus fixes any axis for $b$ in $\mathcal{A}$. Therefore, either $b$ translates orthogonal to any geodesic in $\mathcal{A}$ that joins $P$ with ${ P}^{op}$, or else $b$ translates along a geodesic in $\mathcal{A}$ that joins $P$ with ${ P}^{op}$.
The latter option would contradict Lemma~\ref{l;mahler} since for any $e \in \mathcal{A}$, we have $\Gamma b^n e=\Gamma e \in \G \backslash X$ and yet there is an unbounded $a \in T$ such that the ray determined by $a^n e$ is parallel to the ray determined by $b^n e$ and yet $a^{-n}ua^n \to 1$ either for any $u \in {\bf U}(\bbz [t])$ or for any $u$ in the $\bbz [t]$-points of the unipotent radical of ${\bf P}^{op}$.

\bigskip

The spherical Tits building for $G$ and $X$ is a graph, and the apartment $\mathcal{A}$ corresponds to a circle in the spherical Tits building. Suppose this circle has vertices $P_1,\ldots,P_n$ and edges $Q_1,...,Q_n$ where each $\bfP_i$ is a maximal proper $\lq$-parabolic subgroup of $\bfG$ containing $\bfA$, each $\bfQ_i$ is a minimal $\lq$-parabolic subgroup of $\bfG$ containing $\bf A$, and $\bfP_1=\bfP$. We further assume that mod $n$, the edge $Q_i$ has vertices $P_{i}$ and $P_{i+1}$.

Notice that $\bfU \leq \bfQ_1 \cap \bfQ_n$ since $\bfP=\bfP_1$  contains both $\bfQ_1$ and $ \bfQ_n$. That is, any element of $\bfU (\bbz)$ fixes the edges $Q_1$ and $Q_n$.

Let $\bfU_1$ be the root group corresponding to the half circle that contains $Q_1$ but not $Q_2$, so that  $\bfU_1 \leq \bfU$ but $\bfU_1 \cap \bfQ_2=1$. 
Let $\bfU_n$ be the root group corresponding to the half circle that contains $Q_n$ but not $Q_{n-1}$, so that  $\bfU_n \leq \bfU$ but $\bfU_n \cap \bfQ_{n-1}=1$. 

It follows that $\bfU-\bfQ_i $ has codimension in $\bfU$ at least $1$ for $i=2,n-1$. Since 
 $\bfU(\bbz)$ is Zariski dense in $\bfU$, there is some $u \in \bfU(\bbz) - (\bfQ_{2} \cup \bfQ_{n-1}) $. It follows that $u$ fixes the edges $Q_n$ and $Q_1$, but no other edges in the circle corresponding to $\mathcal{A}$. 
 
 Since $u$ is a bounded element of $G$, it fixes a point in $X$. Therefore, $u$ fixes a geodesic ray in $X$ that limits to an interior point of the edge corresponding to ${Q}_1$ in the spherical building. Any such geodesic ray must contain a point in $\mathcal{A}$, which is to say that $u$ fixes a point in $\mathcal{A}$.
 
 \bigskip
 
 Define a height function $q:\mathcal{A} \rightarrow \bbr$ such that the pre-image of any point is an axis of translation for $b$, such that $s \leq t$ if and only if any geodesic ray in $\mathcal{A}$ that eminates from $q^{-1}(s)$ and limits to $P$ contains a point from $q^{-1}(t)$.

Let $F=\{\, x \in \mathcal{A} \mid ux=x \,\}$, let $I=\inf _{f \in F} \{\, q(f)  \,\}$, and let  $E=\{\,f \in F \mid q(f)=I \,\}$. Since the fixed set of $u$ in the circle at infinity of $\mathcal{A}$ equals the union of the two edges $Q_1$ and $Q_n$, and since $F$ is convex, $I$ exists and $E$ is either a point of, a subray of, a line segment of, or an entire axis of translation for $b$.

Notice that $E$ is bounded, otherwise $u$ would fix the point at infinity that a subray of $E$ limited to. This point at infinity would have distance $\pi /2$ from the vertex $P$ in the spherical metric, but this is not possible as the previously identified fixed set of $u$ in the boundary circle is centered at $P$ and has radius at most $\pi /3$. (The bound $\pi/3$ is realized exactly when the root system for $G$ is of type $A_2$.) Thus $E$ is either a point or a compact interval.

\bigskip

 Since the fix set of $u$ in the boundary circle is exactly the union of ${Q}_1$ and ${ Q}_n$, and since $F$ is convex, $F$ is precisely the union of all geodesic rays eminating from points in $E$ and limiting to points in the arc ${ Q}_1 \cup { Q}_n$. That is $F$ is a polyhedral region in $\mathcal{A}$ that is symmetric with respect to a reflection of $\mathcal{A}$ through a geodesic that limits to ${ P}$ and the opposite point of ${ P}$. If $E$ is a point, then $F$ has two geodesic rays as its boundary: one ray that limits to ${ P}_2$, and the other that limits to  ${ P}_n$. If $E$ is a nontrivial interval, then the boundary of $F$ is the union of $E$, a ray from an endpoint of $E$ that limits to  ${ P}_2$, and a ray from the other endpoint of $E$ that limits to ${ P}_n$.

If $E$ is an interval, we label its endpoints $e^+$ and $e^-$ such that $E$ is both oriented in the direction of translation of $b$, and in the direction towards $e^+$, and away from $e^-$. Let $e_0$ be the midpoint of $E$.
If $E$ is a point, then $e_0=e^+=e^-$ is that point.

 For $n_0$ sufficiently large and for any $n \geq n_0$, we define $\sigma _n \subseteq \mathcal{A}$ as the geodesic segment between $b^{-n}e^+$ and $b^{n}e^-$. Notice that $b^{-n}e^+$ is the only point in $\sigma_n$ that is fixed by $g_n=b^{-n}ub^n$, and that  
$b^{n}e^-$ is the only point in $\sigma_n$ that is fixed by $h_n=b^{n}ub^{-n} $.

\bigskip

Recall that $\mathcal{A}$ is the apartment corresponding to $A$ 
and $\bfT\subset\bfA$ is a $\bbq$-split one dimensional torus of $\bfG.$ Recall also that $\bfP=\bfU Z_\bfG(\bfT).$ Let $a\in T$ be such that  $a^{-n}ua^{n}\to 1$ as $n\to \infty$ so that $a^ne_0$ converges to the cell at infinity corresponding to $ P$ as $n\rightarrow\infty.$

Let $\Delta _n$ be the triangle with one face equal to  $\sigma _n$, a second face contained in the boundary of $b^{-n} {\rm{Fix}}_\mathcal{A}(u)={\rm{Fix}}_\mathcal{A}(g_n)$, a third face contained in the boundary of $b^{n} {\rm{Fix}}_\mathcal{A}(u)={\rm{Fix}}_\mathcal{A}(h_n)$, and vertices $b^ne^-$, $b^{-n}e^+$, and a uniquely determined point $y_n \in  \partial {\rm{Fix}}_\mathcal{A}(g_n) \cap  \partial {\rm{Fix}}_\mathcal{A}(h_n)$. Thus $y_n $ converges to the cell at infinity corresponding to $ P$ as $n\to \infty.$

Note that 
\begin{enumerate}
\item $\bfU$ is a unipotent group so $[[[[g_n, h_n] ,\cdots ],h_n],h_n]=1$ for some fixed number of nested commutators that's independent of $n$.
\item If $w$ is a word in $\{g_n,h_n,g_n^{-1},h_n^{-1}\}$ and $d\in\{g_n,h_n,g_n^{-1},h_n^{-1}\},$ then $w\sigma _n$ and $wd\sigma _n$ are incident.
\end{enumerate} 
(1) and (2) imply that the word $[[[[g_n, h_n], \cdots ],h_n],h_n]$ (or possibly a subword) describes a 1-cycle that is the union of translates of $\sigma _n$ by subwords of $[[[[g_n, h_n], \cdots ],h_n],h_n]$. We name this $1$-cycle $c_n$.

The cone of $c _n$ at the point $y_n$ is the topological image of a $2$-disk $\phi _n: D^2 \rightarrow X$ such that $\phi _n (\partial D^2)=c_n$.

If we let $$X_0=\G\sigma_{n_0}$$ then clearly $c_n \in X_0$ for all $n$ since $b,g_n,h_n \in \G$ and $\sigma _n \subseteq \langle b \rangle \sigma_{n_0}$.

\section{Proof of Theorem~\ref{t;main}}\label{sec;proof}

We choose
a $\G$-invariant and cocompact space $X_i \subseteq X$
to satisfy the inclusions
$$X_0 \subseteq X_1 \subseteq X_2 \subseteq ... \subseteq
\cup_{i=1}^\infty X_i =X$$

In our present context, Brown's criterion takes on the following
form \cite{Brown filtration}

\begin{bc} By Lemma~\ref{l;stab}, the group $\G$ is not of type
$FP_{2}$ (and hence not finitely presented) if for any $i \in \mathbb{N}$, there exists some class
in the homology group $\widetilde{\text{H}}_{1} (X_0 \,,\,
\mathbb{Z})$ which is nonzero in $\widetilde{\text{H}}_{1} (X_i
\,,\, \mathbb{Z})$.
\end{bc}

Since $\G\backslash X_i$ is compact it follows from Lemma~\ref{l;mahler} that for any $i$ there there exists some $j_i$ such that $a^{j_i}e_0\not\in X_i.$ Choose $n$ sufficiently large so that $a^{j_i}e_0 \in \Delta _n \subseteq \phi _n$. Recall that $c_{n}\subseteq X_0$.  Since $X$ is contractible and 2-dimensional, any filling disk for $c_{n}$ must contain $a^{j_i}e_0$. That is, $c_n$ represents a nontrivial class in the homology of $X-\{a^{j_i}e_0 \}$, and hence is nontrivial in the homology of $X_i$.

\section{Other ranks}

The proof of Proposition 4.1 in \cite{BW} gives a short proof that ${\bf{SL_2}}(\mathbb{Z} [t])$ is not finitely generated by examining the action of ${\bf{SL_2}}(\mathbb{Z} [t])$ on the tree for ${\bf{SL_2}}(\mathbb{Q}((t^{-1})))$. Replacing some of the remarks for ${\bf{SL_2}}(\mathbb{Z}[t])$ in that paper with straightforward analogues from lemmas in this paper, it is easy to see that the proof in \cite{BW} applies to show that if  $\bf H$ is a connected, noncommutative, absolutely  almost simple algebraic $\bbq$-group of $\bbq$-rank $1$, then ${\bf H}(\bbz [t])$ is not finitely generated.

It seems natural to state the following

\begin{conjecture}
Suppose $\bf H$ is a connected, noncommutative, absolutely  almost simple algebraic $\bbq$-group whose $\bbq$-rank equals $k$. Then ${\bf H}(\bbz [t])$ is not of type $F_k$ or $FP_k$.
\end{conjecture}

The conjecture has been verified when $K=\mathbb{Q}$ and ${\bf H}={\bf SL_n}$ \cite{BMW}.

%%%%%%%%%%%%%%%%%%%%%%%%%%%%%%%%%%%%%%%%%%%%%%%%%%%%%%%%%%%%%%%%%%%%%%%%%%%

\end{document}